\documentclass[11pt, a4paper]{article}

\usepackage{inputenc} 
\usepackage[english]{babel}
\usepackage{dsfont}

\usepackage[totalheight=24 true cm, totalwidth=17 true cm]{geometry}
\usepackage{enumitem}

\setlist[itemize]{noitemsep}

\usepackage{bbding}
\usepackage{amsmath,multirow}
\usepackage{amsthm}
\usepackage{amssymb}
\usepackage{mathrsfs}
\usepackage{mathtools}
\usepackage{ stmaryrd }
\usepackage{url}
\usepackage{wrapfig}
\usepackage{starfont}
\usepackage{pifont}
\usepackage{eurosym}
\usepackage{subcaption}

\usepackage{algorithm}
\usepackage{algpseudocode}

\usepackage{multirow,booktabs,array,arydshln}% for complicated tables

\usepackage{imakeidx}
\makeindex[]

\usepackage{multicol}
\usepackage{relsize}
\usepackage{float}
\usepackage{tikz,pgf}
\usepackage{tikz-cd}
\usepackage{wasysym}
\usepackage{graphicx}
\usepackage{fancyhdr}
\usepackage{verbatim}

\usepackage{pgfplots}
\pgfplotsset{compat=1.15}
\usepackage{mathrsfs}
\usetikzlibrary{arrows}

\usepackage[all,dvips,knot,web,arc,curve,color,frame]{xy}
\usepackage[all,knot,arc,color,web]{xy}
\xyoption{arc}
\xyoption{web}
\xyoption{curve}

\numberwithin{equation}{section}

\theoremstyle{plain}
\newtheorem{theorem}{Theorem}[section]
\newtheorem{proposition}[theorem]{Proposition}
\newtheorem{corollary}[theorem]{Corollary}

\theoremstyle{definition}
\newtheorem{definition}[theorem]{Definition}

\newtheorem{remark}[theorem]{Remark}

\newtheorem{example}[theorem]{Example}

\usepackage[bookmarksnumbered=true]{hyperref} 
\hypersetup{
     colorlinks = true,
     linkcolor = blue,
     anchorcolor = blue,
     citecolor = teal,
     filecolor = blue,
     urlcolor = blue
     }
\newcommand\restr[2]{{% we make the whole thing an ordinary symbol
  \left.\kern-\nulldelimiterspace % automatically resize the bar with \right
  #1 % the function
  \vphantom{\big|} % pretend it's a little taller at normal size
  \right|_{#2} % this is the delimiter
  }}

\everymath{\displaystyle}
\allowdisplaybreaks

\makeatletter
\def\mathcenterto#1#2{\mathclap{\phantom{#1}\mathclap{#2}}\phantom{#1}}
\let\old@widetilde\widetilde
\def\widetildeto#1#2{\mathcenterto{#2}{\old@widetilde{\mathcenterto{#1}{#2\,}}}}
\let\old@widehat\widehat
\def\widehatto#1#2{\mathcenterto{#2}{\old@widehat{\mathcenterto{#1}{#2\,}}}}
\makeatother

\usepackage{tikz,caption}

\usepackage{verbatim}
\usetikzlibrary{arrows,shapes}
\usetikzlibrary{positioning}

\def\kb{{\mathbf k}}
\def\xb{{\mathbf x}}

\def\Xk{{\mathcal X}}

\def\Fc{{\mathcal F}}
\def\Bc{{\mathcal B}}

\def\p{{\mathbf p}}
\def\x{{\mathbf x}}
\def\R{{\mathbf R}}

\def\cls{{\rm cls}}

\def\md{{\rm md}}

\def\mult{{\rm mult}}

\hypersetup{
    colorlinks=true,
    linkcolor=blue,
    filecolor=magenta,      
    urlcolor=cyan,
    pdftitle={Overleaf Example},
    pdfpagemode=FullScreen,
    }
    
\setlist[itemize]{noitemsep}

\normalfont

\usepackage{bbm}
\usepackage{dsfont}

\usepackage{url}
\usepackage{amsmath,blkarray,booktabs, bigstrut}
\usepackage{tikz}
\usepackage{multirow}
\usepackage{xcolor}
\usepackage{color}
\usepackage[maxbibnames=99]{biblatex}
\title{Stable Coherent Systems}

\author{Rodrigo Iglesias         \and
        Fatemeh Mohammadi \and
        Patricia Pascual-Ortigosa \and
        Eduardo S\'aenz-de-Cabez\'on \and
        Henry P. Wynn}

\begin{document}
\maketitle

\begin{abstract}
We describe the notion of stability of coherent systems as a framework to deal with redundancy. We define stable coherent systems and show how this notion can help the design of reliable systems. We demonstrate that the reliability of stable systems can be efficiently computed using the algebraic versions of improved inclusion-exclusion formulas and sum of disjoint products.
%\keywords{System reliability\and stable systems\and algebraic reliability\and inclusion-exclusion\and sum-of-disjoint-products.}
%\subclass{90B25 \and 13F55}
\end{abstract}

{\hypersetup{linkcolor=black}
{\tableofcontents}}

\vspace{-2cm}
\section*{Special notations}
\begin{center}
\begin{minipage}{9cm}
\begin{enumerate}
\item[$S$] A coherent system
\item[$\{c_1,\dots,c_n\}$] The set of components of system $S$.
\item[$s=(s_1\dots,s_n)$] Tuple indicating the states of the components of a system.
\item[$s^{i}_-$] Tuple with the same values as $s$ except $s_i=s_i-1$.
\item[$s^{i}_+$] Tuple with the same values as $s$ except $s_i=s_i+1$.
\item[$s^{(i,j)}_{(-,+)}$] Tuple with the same values as $s$ except $s_i=s_i-1$ and $s_j=s_j+1$.
\item[$\p$] Vector of working probabilities of the components of a binary system.
\item[$pr(s)$] $pr(c_1\geq s_1 \wedge c_2\geq s_2\wedge \cdots \wedge c_n\geq s_n)$.
\item[$p_{i,j}$] The probability $pr(c_i\geq j)$ that component $c_i$ is performing at least at level $j$, that is $pr(s)$ for $s=(0,\ldots,j,\ldots,0)$. %(In the main text we can write that this is $pr(c_1\geq 0,\ldots,c_{i-1}\geq 0, c_i\geq j,c_{i+1}\geq 0,\ldots,c_n\geq 0)$.
\item[$pr(x^s)$] $pr(s)$.
\item[$R_l(S)$] The $l$-reliability of system $S$, i.e.  probability that the system $S$ is performing at level $\ge l$.
\item[$R_\p(S)$] The reliability of the binary system $S$ using $\p$ as the probability vector of its components.
\item[ $\xb^{(\overline{i,j})}$] The vector of states $(x_1,\dots,x_{i-1},0,x_{i+1},\ldots,x_{j-1},0,x_{j+1},\ldots,x_n)$.
\item[$I_{S,l}$] The $l$-reliability ideal of system $S$.
\item[$HN_{I_{S,l}}$] Numerator of the Hilbert series of $I_{S,l}$.
\item[$\Xk_{L,\mathcal{N}}(\nu)$] Set of multiplicative indices of $\nu$ with respect to the involutive division $L$ and set $\mathcal{N}$.
\item[$\overline{\Xk}_{L,\mathcal{N}}(\nu)$] Set of non-multiplicative indices of $\nu$ with respect to the involutive division $L$ and set $\mathcal{N}$.

\end{enumerate}
\end{minipage}
\end{center}

\section{Introduction}
With the increase in complexity and external risk to modern systems the role of backup has increased its importance. Examples are essential databases in finance and many commercial fields, power and communication networks, medical supplies and supply chains more generally. But backing up is expensive and for this reason there is a need to develop measures of value of backup and standby in order to design more reliable systems. On the other hand, redundancy is one of the driving forces in the design of coherent systems. The balance between redundancy and cost optimization is a main criterion for the reliability-based design of coherent systems, e.g. a series system is cheap but not redundant/reliable, while a parallel system is on the contrary very redundant, but also very demanding in terms of resources.  In between are, for example, series-parallel systems and $k$-out-of-$n$:G systems, where G is for {\em good}, which work whenever at least $k$ of its $n$ components are working (as opposed to $k$-out-of-$n$:F systems, where F is for {\em failure}, which fail whenever at least $k$ of their $n$ components fail). Unless otherwise stated we will always consider $k$-out-of-$n$:G systems and denote them simply as $k$-out-of-$n$. 

In this paper we study systems, which we will call {\em stable}, that have good backup features, and propose stable systems as a kind of systems that share some properties with the usual $k$-out-of-$n$ model. In particular, stable systems generalize $k$-out-of-$n$ systems' notion of redundancy. The idea behind stable systems is that the improvement of some components can compensate the degradation of others. We use the term {\em backup} to describe this process.
The seminal paper of Birnbaum \cite{birnbaum1968some} introduced the idea of importance measures of a component in a system in terms of the sensitivity of the failure of the whole system to failure (or removal) of the component.  They are sometimes called “fault indices” or “criticality indices”. Kuo and Zhu,
in \cite{kuo2012importance} give a comprehensive review of such measures.

We consider classical ideas of coherency and redundancy combined with the recent area of {\em algebraic reliability}, with which the authors have been involved for a number of years.  At a general level the ambition of the paper is that algebraic reliability can be used to help meet the demand for theoretical background for maintenance aspects of reliability and cost vs. benefit analysis.
The outline of the paper is as follows. In Section \ref{sec:stableSystems} we give the main definitions of stable systems. We study design issues and some importance measures for these systems in Section \ref{sec:importanceMeasures}. In the rest of the paper we perform an algebraic analysis of the reliability of stable systems. For this, we first present a brief review of the algebraic method in Section \ref{sec:algebraic} and apply it to stable systems in Section \ref{sec:algebraicStable}. Finally, in Section \ref{sec:experiments} we present the results of some computational experiments and simulations that demonstrate the efficiency of the algebraic methods when dealing with stable systems.

\section{Fully stable, strongly stable and stable coherent systems}\label{sec:stableSystems}
%In what follows we should be clear that the system itself has states, which are described by allocation of states to the individual components.

A system $S$ consists of a set of $n$ components and a structure function. The system's levels of performance are given by a discrete set $L=\{0,1,\dots, M\}$ indicating growing levels of performance, the system being in level $0$ indicates that the system is failing, and level $j>i$ indicates that the system is performing at level $j$ better than at level $i$.   We denote the system's components by $c_i$ with $i\in\{1,\dots,n\}$. Each of the individual components $c_i$ can be in one of a discrete set of levels $L_i=\{0,\dots,M_i\}$, also called states. We say that a state of the system is the $n$-tuple of its components' states at a particular moment in time. Given two system states $s=(s_1,\dots,s_n)$ and $t=(t_1,\dots,t_n)$, we say that $s\geq t$ if $s_i\geq t_i$ for all $i=1,\dots,n$ and that $s\leq t$ if $s_i\leq t_i$ for all $i$. For ease of notation, states of the systems can be represented in monomial form, in which a state $s=(s_1,\dots,s_n)$ is represented by the monomial $c_1^{s_1}\cdots c_n^{s_n}$. The structure function of the system, denoted by $\Phi$, describes the level of performance of the system in terms of the states of its components, i.e. $\Phi:L_1\times \cdots \times L_n\longrightarrow L$. The system $S$ is {\em monotone} if $\Phi$ is non-decreasing; if in addition each component is relevant to the system then the system is said to be {\em coherent}. A component is said to be relevant  for system $(S,\Phi)$ if its status (level of performance) does affect the system state. In this paper we consider all systems to be coherent, although for all our results it is enough that $S$ is monotone. A system $S$ with a structure function $\Phi$ will be denoted by $(S,\Phi)$; if the structure function is clear from the context we will simply refer to the system as $S$.

The description of the $l$-th level of a system is given by its set of $l$-working states (also called $l$-paths), i.e. those tuples $(s_1,\dots,s_n) \in L_1\times\cdots\times L_n$ such that $\Phi(s_1,\dots,s_n)\geq l$. A state $(s_1,\dots,s_n)\in L_1\times\cdots\times L_n$ is a {\em minimal $l$-working state}, also called {\em minimal $l$-path}, if  $\Phi(s_1,\dots,s_n)\geq l$ and $\Phi(t_1,\dots,t_n)< l$ whenever all $t_i\leq s_i$ and the inequality is strict in at least in one case. A state $(s_1,\dots,s_n)\in L_1\times\cdots\times L_n$ is a {\em minimal $l$-failure state} or {\em minimal $l$-cut} if  $\Phi(s_1,\dots,s_n)< l$ and $\Phi(t_1,\dots,t_n)\geq l$ whenever all $t_i\geq s_i$ and at least one of the inequalities is strict. We denote by $\Fc_l(S,\Phi)$ the set of $l$-paths of $S$ with respect to the structure function $\Phi$; the set of minimal $l$-paths will be denoted by $\overline{\Fc}_l(S,\Phi)$.  If the structure function is clear from the context, we simply write $\Fc_l(S)$ and $\overline{\Fc}_l(S)$.

For a component $c_i$ and a level $j\in L_i$,
we denote by $p_{i,j}$ the probability that  $c_i$ is performing at least at
level $j$, i.e. $p_{i,j}=pr(c_i\geq j)$; this probability is also called {\em $j$-reliability} of the component. Given a tuple $s=(s_1,\dots,s_n) \in L_1\times\cdots\times L_n$ we denote by $pr(s)$ the probability of the conjunction $(c_1\geq s_1)\wedge\cdots \wedge(c_n\geq s_n)$. In the case of binary components the usual notation is $p_i=pr(c_i=1)$ and $q_i=1-p_i=pr(c_i=0)$. 
Under the assumption of independent components, $pr(s)=\prod_{1\leq i\leq n}p_{i,s_i}$.
The $l$-reliability of $S$, denoted by $R_l(S)$, is the probability that the system $S$ is performing at level $\ge l$; conversely, the $l$-unreliability of $S$, denoted $U_l(S)$, is $1-R_l(S)$. The $l$-reliability of $S$ can be expressed as the probability of the disjunction
\[
R_l(S)=pr\left( \bigvee_{s\in{\Fc}_l(S)}(c_1\geq s_1)\wedge\cdots \wedge(c_n\geq s_n)\right),
\]
i.e. the probability that the system is in a state that is bigger than at least one $l$-path. In binary systems and components, the $1$-reliability of component $i$ is simply called {\em reliability} and is denoted by $p_i$; and the $1$-reliability of the system is simply called {\em reliability} of the system, and denoted by $R(S)$.
%Thus, let $S$ be a multi-state coherent system with $n$ components and let $\Phi$ be its structure function. The system can be in a discrete number of performing or working states $0,\dots,M$. If the system is at level $l\in\{1,\dots,M\}$ we say the system is working at level $l$. The $l$-reliability of $S$ is the probability that the system is working at level $l$ or greater. Each component $c_i$ of $S$ can be performing at a number of different levels $0,\dots,M_i$. The structure function of $S$ describes the performance level of the system in terms of the performance level of its components.

\begin{definition}
We say that a system $S$ with structure function $\Phi$ is \emph{fully stable} for level $l$ if for any $l$-path $m\in\Fc_l(S,\Phi)$, $m=(m_1,\dots,m_n)$ and any component $i \in \{1,\dots,n\}$ we have that $m^{(i,j)}_{(-,+)}=(m_1,\dots,m_i-1,\dots,m_j+1,\dots,m_n)$ is also an $l$-path of $S$ under the structure function $\Phi$ for any $i\neq j\in\{1,\dots,n\}$.
\end{definition}

In other words, a system is said to be fully stable if whenever the system is in an $l$-working state and any component suffers a one-level degradation, the system can be kept performing at level $l$ by the one-level improvement of any other component. Any component serves as a backup for any other.  In the binary case it is easy to see that the only fully stable systems are $k$-out-of-$n$ systems. Although stability is a property of the structure function, for simplicity we also consider it as a property of the system. We do the same in the next definitions.

In this paper we will focus on the next two definitions, which relax the conditions of fully stable systems to describe less demanding versions of stability. For them, we need to set an ordering $\tau$ of the components. Unless otherwise stated we assume that  $\tau$ is the usual ordering $1<2<\cdots<n$.

\begin{definition}
We say that a system $S$ with structure function $\Phi$ is \emph{strongly stable} for level $l$ if there exists an ordering of the components $\tau$ such that for any $l$-path $m\in\Fc_l(S,\Phi)$, $m=(m_1,\dots,m_n)$ and any component $i \in \{1,\dots,n\}$ we have that $m^{(i,j)}_{(-,+)}=(m_1,\dots,m_j+1,\dots,m_i-1,\dots,m_n)$ is also an $l$-path of $S$ under the structure function $\Phi$ for any $i >_{\tau} j\in\{1,\dots,n\}$. We say that $S$ is strongly stable if it is strongly stable for all its levels.
\end{definition}

\begin{definition}
We say that a system $S$ with structure function $\Phi$ is \emph{stable} or \emph{simply stable} for level $l$ if there exists an ordering of the components $\tau$ such that for any $l$-path $m\in\Fc_l(S,\Phi)$, $m=(m_1,\dots,m_i,0,\dots,0)$, such that $i$ is the last (with respect to $\tau$) working component (i.e. not in total failure) of $m$, we have that $m^{(i,j)}_{(-,+)}=(m_1,\dots,m_j+1,\dots,m_i-1,0,\dots,0)$ is also an $l$-path of $S$ under the structure function $\Phi$ for any $i >_{\tau} j\in\{1,\dots,n\}$. We say that $S$ is stable if it is stable for all its levels.
\end{definition}

In strongly stable systems, any component can be backed up by any other component with a smaller index with respect to the ordering $\tau$. For stable systems we only demand that within each working path, the last component of the path is backed up by the rest of the components with smaller indices. In this paper, we will usually consider strongly stable systems.

\begin{remark}
For binary systems we need to introduce a subtle modification in these definitions. In both the strongly stable and simply stable case, the backup components must be in a failure state. If a component is already in its working state, it cannot be further improved to serve as a backup for a failed component. %If a component is already in its working state, it cannot be further improved to back a failed component up. 
%This restriction is not necessary in multi-state systems.
\end{remark}

%\begin{remark}
%Stability and strong stability are properties that can be expressed in terms of the ordering of the states of the system. In our case, we use the ordering given by $s\leq t$ if $s_i\leq t_i$ for every $i$, i.e. componentwise ordering. with respect to $\prec$ is an ideal such that $\alpha\prec \beta$ and $\alpha\in I$ imply that $\beta\in I$. Note that the standard definition is when $\prec$ is the total degree lexicographic ordering (deglex). We call an ideal squarefree strongly stable ideal if we only restrict to squarefree lattice points. 
%\end{remark}

Let $S$ be a coherent system with $n$ components and let $\Phi$ be its structure function. Consider a fixed ordering of the components of $S$. We say that a structure function $\Psi$ on $S$ dominates $\Phi$ at level $l$ if every $l$-path of $S$ with respect to $\Phi$ is also an $l$-path with respect to $\Psi$. We define the {\em stable closure} of $\Phi$ for level $l$ as a structure function $\Psi$ such that it dominates $\Phi$ at level $l$, is stable and such that any other stable structure function $\Psi'$ that dominates $\Phi$ also dominates $\Psi$ at level $l$. In the same way we define strongly stable and fully stable closures. Observe that the fully stable closure of a binary system $S$ is the $k$-out-of-$n$:G system, where $k$ is the minimal length of any minimal path of $S$.

\begin{example}
Let $n=5$ and let $S$ be a binary coherent system with five components whose set of minimal paths with respect to the structure function $\Phi$ is $\overline{\Fc}(S,\Phi)=\{c_1c_4,c_2c_3c_4,c_2c_4c_5\}$.

The function $\Psi$ such that the set of minimal paths for $S$ is 
\[
\overline{\Fc}(S,\Psi)=\{c_1c_2,c_1c_3, c_1c_4,c_2c_3c_4,c_2c_3c_5,c_2c_4c_5\}
\]
is the strongly stable closure of $(S,\Phi)$.

The function $\Psi'$ such that the set of minimal paths for $S$ is 
\[
\overline{\Fc}(S,\Psi')=\{c_1c_2,c_1c_3,c_1c_4,c_2c_3c_4,c_2c_4c_5\}
\]
is the stable closure of $(S,\Phi)$.
\end{example}

In Section~\ref{sec:algebraicStable}, we describe an algorithm to obtain the stable and strongly stable closure of a system with a given structure function.

\section{Design of stable systems based on components' importance}\label{sec:importanceMeasures}
The fact that the definitions of stability and strong stability depend on the ordering of the components raises the issue of how stable orderings relate to other features of the system, in particular to importance measures of its components. 

In binary and multi-state systems, importance measures are used to calculate the relative importance of their components for the overall performance of the system, cf.  \cite{KZ12}. The role of these measures is manifold, in particular they provide a ranking of the components with respect to their influence on the system's reliability and help to focus on the top contributors to system reliability and unreliability, and on improvements with the greatest reliability effect. Two main properties of the system are considered to analyze the importance of each of the components: its structure and its reliability. Correspondingly, we study the positions of each of the components in the system's structure, and their contribution to the system's reliability. Importance measures based on the position of the components in the systems are called {\em structural} importance measures, while those taking into account the reliability of the system are called {\em reliability} importance measures. In applied reliability studies of complex systems, one of the most time-consuming tasks is to find good estimates for the failure and repair rates. In systems with a big number of components one may start with coarse estimates, calculate measures of importance for the various components, such as Birnbaum’s or structural measures, and spend most of the time finding higher-quality data for the most important components. Components with a very low value of Birnbaum’s or other importance measures will have a small effect in the system reliability, and extra efforts finding higher quality data for such components may be considered wasted. The main importance measures have been developed and studied for binary systems \cite{KZ12} but there exist certain generalizations and extensions to multi-state systems. These generalizations can be classified as those that measure the contribution of a component to the performance of the system and those that measure which states of a given component are more important to the performance of the system \cite{RC05,RRGCT06}.

Birnbaum importance or $B$-importance for binary systems \cite{B69} is reliability based, it considers the probability that a component is critical for the system. It can be used to define other importance measures and is often used for comparisons among importance measures. The acronym $B$-i.i.d importance refers to the cases in which all components of the system are {\em independent} and their failures are statistically {\em identically distributed}. It thus reflects the structural aspects of this importance measure. When one considers the $B$-i.i.d importance with $p=1/2$ one has the so called $B$-structural importance \cite{KZ03}. It is defined as
\[
I^S_i=\frac{1}{2^{n-1}}\sum_{s} [\Phi(1_i,s)-\Phi(0_i,s)],
\]
where the sum is over all states $s$ of the system, and $(1_i,s)$ indicates the tuple $s$ setting to $1$ its $i$-th component, and $(0_i,s)$ indicates that we set to $0$ the $i$-th value of $s$.

If the system $S$ is strongly stable with respect to some ordering $\tau$ of the variables, then $\tau$ sorts the variables decreasingly by their structural importance. We will see an algebraic proof of this fact in Section \ref{sec:algebraicImportance}. Using stability as a criterion for redundancy is therefore compatible with design based on structural importance.

Strongly stable systems have their components also ordered by {\em permutation importance}, a structural importance measure defined for binary systems in \cite{BPT89} and extended in \cite{M96} to the multi-state case. We prove here the result for multi-state systems, but the result applies verbatim to the binary case.

\begin{definition}
Component $i$ is more permutation important than component $j$ in a multi-state coherent system if 
\[
\Phi(x_1,\dots,\overset{i}{\mu},\dots,\overset{j}{\nu},\dots,x_n)\geq \Phi(x_1,\dots,\overset{i}{\nu},\dots,\overset{j}{\mu},\dots,x_n)
\]
 for all $0\leq\nu<\mu\leq \min(M_i,M_j)$ and all $\xb^{(i,j)}$, and strict inequality holds for some $\nu<\mu$ and $\xb^{(i,j)}$, where $M_i$ and $M_j$ denote the maximum performance level of components $i$ and $j$ respectively.
\end{definition}

\begin{proposition}\label{prop:stablePermutationImportance}
Let $S$ be a strongly stable multi-state system with respect to order $\tau$. Then component $i$ is at least as  permutation important as component $j$ whenever $i<_\tau j$.
\end{proposition}
\begin{proof}
Let $\nu<\mu$ and let $s=(s_1,\dots,\overset{i}{\nu},\dots,\overset{j}{\mu},\dots,s_n)$ an $l$-path of $S$. Since $S$ is strongly stable, we have that $s^{(j,i)}_{(-,+)}$ is also an $l$-path of $S$. We can proceed iteratively to degrade component $j$ while improving component $i$ and reach the state $s'=(s_1,\dots,\overset{i}{\mu},\dots,\overset{j}{\nu},\dots,s_n)$ which is still an $l$-path. In particular, we have that 
\[
\Phi(s_1,\dots,\overset{i}{\mu},\dots,\overset{j}{\nu},\dots,s_n)\geq\Phi(s_1,\dots,\overset{i}{\nu},\dots,\overset{j}{\mu},\dots,s_n),
\]
hence the permutation importance of component $i$ is greater than or equal to that of component $j$.
\end{proof}
\begin{remark}
Observe that the opposite to Proposition \ref{prop:stablePermutationImportance} does not always hold, i.e. if $i>_\tau j$ then we cannot claim that the permutation importance of component $i$ is bigger than that of component $j$. Consider for example a multi-state system with three components whose minimal $l$-paths are given by
\begin{align*}
\{ (2,0,0), (1,2,0), (1,1,2), (0,4,0),(0,3,1),
\\(0,2,2),(1,0,3),(0,1,3),(0,2,3),(0,0,4)\}
\end{align*}
for some $l$. The system is strongly stable for level $l$, and the permutation importance of component $1$ is strictly greater than that of component $3$. To see this, observe that for any $l$-path in which the state of component $3$, say $\mu$, is bigger than that of component  $1$, say $\nu$, we have that the vector in which the states of components $1$ and $3$ are interchanged is still a (possibly non-minimal) $l$-path of the system. However, for the state $(2,0,0)$ the state $(0,0,2)$ that results from interchanging the state of components $1$ and $3$ is not an $l$-path of the system, which implies that the permutation importance of component $3$ is less than that of component $1$.
\end{remark}

For binary systems, we have an analogous result that takes into account the reliability of each of the components. We can see that if the system is strongly stable, then the variables are sorted by their contribution to the reliability of the system. Let $\p=(p_1,\dots,p_n)$ denote the vector of working probabilities of the components of a binary coherent system $S$, where $p_i=pr(c_i>0)$. We denote by $R_\p(S)$ the reliability of $S$ using $\p$ as the probability vector of its components. 

\begin{proposition}
Let $S$ be a strongly stable binary system with respect to the usual ordering of the components $1<2<\cdots<n$. Let $j<i$ and $\p'=\{p_1,\dots,\overset{j}{p_i},\dots,\overset{i}{p_j},\dots,p_n\}$ (i.e. we interchange the working probabilities of components $i$ and $j$). Then, 
\[
p_i\geq p_j \implies R_{\p'}(S)\geq R_{\p}(S).
\]
\end{proposition}
\begin{proof}
We have that 
\[
R(S)=pr\left( \bigvee_{s\in{\Fc}(S)}(c_1\geq s_1)\wedge\cdots \wedge(c_n\geq s_n)\right),
\]
where $\Fc(S)=\{f_1,\dots,f_r\}$ is the set of working states of $S$ and the probabilities are computed using $\p$. Let $j<i$ such that $p_i<p_j$ and $\p'$ results from interchanging $p_i$ and $p_j$ in $\p$. %<-- should we remove this last sentene?
Now, consider the following cases for any path $f\in\Fc(S)$:
\begin{itemize}
\item[-]{If neither $i$ nor $j$ are in $f$ then $pr_\p(f)=pr_{\p'}(f)$.}
\item[-]{If both $i$ and $j$ are in $f$ then $pr_\p(f)=pr_{\p'}(f)$.}
\item[-]{If $i\in f$ but $j\notin f$, then since $S$ is strongly stable, there exists some $f'\in\Fc(S)$ such that $f'=f^{(i,j)}_{(-,+)}$, hence $pr_\p(f)=pr_{\p'}(f')$, $pr_{\p'}(f)=pr_{\p}(f')$, and hence the computation of the total reliability is unaltered.}
\item[-]{If $j\in f$ and $i\notin f$, and $f^{(j,i)}_{(-,+)}\in \Fc(S)$ then the same argument from the previous case holds; hence, the total reliability is unaltered.}
\item[-]{If $j\in f$ and $i\notin f$, but $f^{(j,i)}_{(-,+)}\notin \Fc(S)$ then $pr_{\p'}(f)\geq pr_\p(f)$.}
\end{itemize}
\end{proof}

This proposition indicates that the reliability of a strongly stable system is higher when the components with smaller indices are the most reliable ones. In fact, ordering the components in a descending order with respect to their working probabilities is the optimal assignment for strongly stable systems.

\begin{corollary}\label{cor:precedence}
For any strongly stable system, the maximum of $R_{\sigma(\p)}(S)$ for any permutation $\sigma$ of $\p$ is attained when $\sigma$ is a  monotone descending ordering of the elements of $\p$.
\end{corollary}
\begin{example}\label{ex:fourcomponents}
Let $S$ be a binary system with $4$ components, $a,b,c,d$ such that its set of minimal paths is $\Fc(S)=\{ab,ac,bc,cd\}$. Observe that $S$ is strongly stable with respect to the orderings $c<b<a<d$ and $c<a<b<d$. Assume that the probabilities we can assign to each of the components are $\{0.6, 0.7, 0.8, 0.9\}$. Then using Corollary \ref{cor:precedence} we can assign $p_a=0.7, p_b=0.8, p_c=0.9,p_d=0.6$ or $p_a=0.8, p_b=0.7, p_c=0.9,p_d=0.6$ to maximize the reliability of the system. Any other distribution of these probabilities results in a less or equally reliable system.

For this system we have that if we consider the structural importance,
\[
I^S_c>I^S_a=I^S_b>I^S_d.
\]
Observe that $S$ is stable with respect to both $\tau=c<a<b<d$ and $\tau'=c<b<a<d$. The role of components $a$ and $b$ is the same with respect to stability, and on the other hand, their structural importance is the same, i.e. they are interchangeable with respect to both criteria.
\end{example}

\section{The algebraic method for system reliability}\label{sec:algebraic}
The algebraic approach to system reliability based on monomial ideals started in \cite{GW04} and was developed in a series of papers, see \cite{SW09,SW10,SW15} among others. The main idea of this approach is to associate to each level $l$ of an $n$-components coherent system $(S,\Phi)$ a monomial ideal $I_{S,l}$ whose monomial set consists of those corresponding to the $l$-working states of $S$ and their multiples. These ideals represent an algebraic encoding of the structure function $\Phi$ of $S$.  A principal contribution of this approach is the construction of improved inclusion-exclusion (IIE) formulas that provide also Bonferroni-type \cite{D03} upper and lower bounds for the reliability of the system. These bounds are based on computing free resolutions of the ideals $I_{S,l}$. Another recent variant of the algebraic method is to obtain a disjoint decomposition of each ideal $I_{S,l}$ such that the $l$-reliability of $S$ is obtained as a sum of disjoint products (SDP), this is based on computing involutive bases of the ideals $I_{S,l}$ \cite{IPS22}.

The ideals $I_{S,l}$ are defined as follows. Let $(S,\Phi)$ be a system with $n$ components. Consider a polynomial ring on $n$ variables over a field $\kb$ (usually $\mathbb{Q}$ or $\mathbb{R}$ are considered in applications), this ring is denoted by $\R=\kb[x_1,\dots,x_n]$.  For any state $s=(s_1,\dots,s_n)$ of $S$ we say that the monomial corresponding to $s$ is $\x^s=x_1^{s_1}\cdots x_n^{s_n}$. The monotonicity of $S$ implies that the monomials corresponding to the set $\Fc_l(S)$ of $l$-paths of $S$ generate a monomial ideal $I_{S,l}\subset \R$ for each level $l$ of $S$. The minimal generating set of $I_{S,l}$ is given by the monomials corresponding to the minimal paths of $S$, i.e. $\overline{\Fc}_l(S)$. The algebraic analysis of these ideals provides information about the system $(S,\Phi)$, such as its reliability. To obtain the reliability of $S$ we assign to each monomial $\x^\mu=x_1^{\mu_1}\cdots x_n^{\mu_n}$ the probability of its correspondent state, i.e. $pr(\x^\mu)=pr(\mu)=pr(c_1\geq\mu_1\wedge c_2\geq\mu_2\wedge\cdots\wedge c_n\geq\mu_n)$. %Where the probability assignments to each of the components are denoted by $p_{i,j}=pr(c_i\geq j)$. 
%In the case of binary components the usual notation is $p_i=pr(c_i=1)$ and $q_i=1-p_i=pr(c_i=0)$. 

\subsection{Improved Inclusion-Exclusion formulas}
The Hilbert function of an ideal $I$ is an integer function that for any $z\in\mathbb{Z}$ gives the number of monomials of degree $z$ that are in $I$, its generating function is called the Hilbert series of $I$. The Hilbert function and the Hilbert series provide a compact method to enumerate the monomials in the ideal $I$. When applied to the $l$-reliability ideal of a system $S$, they enumerate the $l$-working states of $S$ and can therefore be used to compute the reliability of the system. Note that in this context we restrict ourselves to resolutions of monomial ideals. 

As can be seen in detail in \cite{SW10,MPSW20}, the Hilbert series of $I_{S,l}$ is a rational function, and its numerator, denoted $HN_{I_{S,l}}$, provides a compact formula for the $l$-reliability of $S$, closely related to the Inclusion-Exclusion formula. If the Hilbert series numerator is furthermore given in the form obtained from a so called {\em free resolution} of $I_{S,l}$ then this formula can be truncated to obtain Bonferroni-like bounds in a compact way \cite{D03}. Therefore, the main ingredient to obtain these algebraic IIE formulas is a free resolution of $I_{S,l}$. The following is a brief description of this important object.

%One way to obtain the multigraded Hilbert series of a monomial ideal~$I$ is by constructing a multigraded free resolution of $I$ and read $HN(I)$ from the data in the resolution. 

The $\R$-module structure of an ideal $I\subseteq \R=\kb[x_1,\dots,x_n]$ is usually described using a {\em free resolution}, which is a series of graded or multigraded free modules and morphisms among them. %A free module is a direct sum of copies of $R$ with the usual grading shifted by some degree $d\in\mathbb{N}$ denoted by $R(-d)$. In the case of monomial ideals we can also have multigraded resolutions, in which the degree shifts are given by multi-degrees $\mu\in\mathbb{N}^n$ and the shifted copies of $R$ are denoted by $R(-\mu)$. A multigraded free resolution of a monomial ideal $I$ is of the form
   \begin{align*}
&0\longrightarrow \bigoplus_{j=1}^{r_d}\R(-\mu_{d,j})\stackrel{\partial_d}{\longrightarrow} \cdots \stackrel{\partial_1}{\longrightarrow}\bigoplus_{j=1}^{r_0}\R(-\mu_{0,j})\stackrel{\partial_0}{\longrightarrow}\R{\longrightarrow}\R/I\longrightarrow 0.
\end{align*}
Here, the $r_i$ are called {\em ranks} of the modules in the free resolution and the $\mu_{i,j}$ for each $i$ denote the multi-degrees of the pieces of the $i$-th module of the resolution. The length of the resolution is given by $d$. Among the various resolutions of an ideal $I$ the {\em minimal free resolution} is the one having smallest ranks; in this case $d$ is known to be less than $n$. The ranks of the minimal free resolution of $I$ are called the {\em Betti numbers} of $I$. %and are a fundamental invariant of $I$ \cite{E95}.

We can now obtain a formulation of $HN_I(x_1,\dots,x_n)$ by means of the descriptors of any (non-necessarily minimal) free resolution of $I$:

\begin{equation}\label{eq:HilbertFormula}
HN_I(x_1,\dots,x_n)=\sum_{i=0}^d(-1)^i\sum_{j=1}^{r_i}\xb^{\mu_{i,j}}. 
\end{equation}

This expression gives a (compact) formula for the $l$-reliability of $S$  if we replace each $\xb^{\mu_{i,j}}$ by $pr(\mu_{i,j})$; it can be truncated and produces the following Bonferroni-type bounds for $R_l(S)$, see \cite{SW09}
 \begin{equation}
     \begin{aligned}\label{eq:resolutionBounds}
       R_l(S)\leq\sum_{i=0}^t(-1)^{i}\sum_{j=1}^{r_i}pr({\mu_{i,j}})\mbox{ for } t\leq d \mbox{ odd,}\\
       R_l(S)\geq\sum_{i=0}^t(-1)^{i}\sum_{j=1}^{r_i}pr({\mu_{i,j}})\mbox{ for } t\leq d \mbox{ even.}\\
     \end{aligned}   
 \end{equation}
 
Free resolutions exist for any monomial ideal and can be constructed in several ways. The resolution producing the most compact algebraic IIE formulas and tighter bounds is the minimal one (which is unique up to isomorphisms), see \cite{SW10} for details on free resolutions and their applications to system reliability. This is in general a demanding computation, although there exist good algorithms that make this approach applicable in practice, see Section \ref{sec:experiments} for further details on computations. 

 \begin{example}[Example 1.4 in \cite{IPS22}]\label{ex:bridge}
 Consider the source-to-terminal network in Figure \ref{fig:bridge}. The minimal paths of this binary system are $MP_1=x_1x_2$, $MP_2=x_4x_5$, $MP_3=x_1x_3x_5$ and $MP_4=x_2x_3x_4$. Its reliability ideal is 
 \[
 I_S=\langle x_1x_2,x_4x_5,x_1x_3x_5,x_2x_3x_4\rangle.
 \]
Using the minimal free resolution of $I(S)$ we obtain the following expression for the reliability of $S$ from the numerator of the Hilbert series of $I(S)$ under the assumption of independent probabilities for the components of $S$:
 \begin{align*}
  R(S)=&p_1p_2{+}p_4p_5{+}p_1p_3p_5{+}p_2p_3p_4\\
  &-(p_1p_2p_3p_4{+}p_1p_2p_3p_5{+}p_1p_2p_4p_5{+}p_1p_3p_4p_5{+}p_2p_3p_4p_5)\\
 &+2 p_1p_2p_3p_4p_5,
 \end{align*}
 while the usual Inclusion-Exclusion formula has the form 
 \begin{align*}
 R(S)=&p_1p_2{+}p_4p_5{+}p_1p_3p_5{+}p_2p_3p_4\\
 &{-}(p_1p_2p_3p_4{+}p_1p_2p_3p_5{+}p_1p_2p_4p_5{+}p_1p_3p_4p_5{+}p_2p_3p_4p_5\\
 &+p_1p_2p_3p_4p_5)+4 p_1p_2p_3p_4p_5-p_1p_2p_3p_4p_5.
 \end{align*}
 
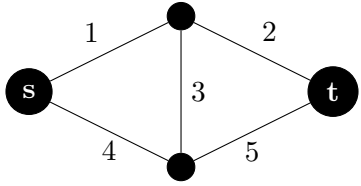
\begin{figure}
\begin{center}
\begin{tikzpicture}
      \node[circle,draw=black, fill=black,minimum size=2pt, text=white] (c1) at (0,0) {$\bf{s}$};      
      \node[circle,draw=black, fill=black,minimum size=1pt] (c2) at (2,1) {};
      \node[circle,draw=black, fill=black,minimum size=1pt] (c3) at (2,-1) {};
      \node[circle,draw=black, fill=black,minimum size=2pt, text=white] (c4) at (4,0) {$\bf{t}$};

      \draw [-] (c1) -- node[auto] {$1$} (c2);
      \draw [-] (c1) -- node[below] {$4$} (c3);
      \draw [-] (c2) -- node[auto] {$3$} (c3);
      \draw [-] (c2) -- node[auto] {$2$} (c4);
      \draw [-] (c3) -- node[below] {$5$}  (c4);
    \end{tikzpicture}
    
\caption{Bridge network.}\label{fig:bridge}
\end{center}
\end{figure}
\end{example}

\subsection{Algebraic Sum of Disjoint Products}
The reliability of networks and other systems have been traditionally evaluated using boolean algebra formulations for the minimal paths (or cuts). The Sum of Disjoint Products approach to system reliability starts with a Boolean product that corresponds to the paths of the system and transforms this expression into another one in terms of disjoint (mutually exclusive) products. Several efficient algorithms have been described in the literature to compute sums of disjoint products, and also several versions of this approach have been developed for multi-state systems, see for example \cite{A79,LT98,XFQD12,HCL22,Y15}. As a simple example consider a system $S$ with three components such that its minimal paths are $\{c_1c_2,c_3\}$. The boolean formulation of the reliability of this system is 
\[
R(S)=pr\left((c_1\wedge c_2)\vee c_3\right).
\]
Which, using inclusion-exclusion can be evaluated as 
\[
R(S)=pr(c_1)pr(c_2)+pr(c_3)-pr(c_1)pr(c_2)pr(c_3)=p_1p_2+p_3-p_1p_2p_3.
\]
The Sum of Disjoint Products formula corresponding to the reliability of this system is
\[
R(S)=pr(c_1)pr(c_2)pr(\bar{c_3})+pr(c_3)=p_1p_2(1-p_3)+p_3,
\]
where $\bar{c_3}$ indicates the complement (negation) of $c_3$, i.e. $pr(\bar{c_3})=1-pr(c_3)$.

An algebraic version of the Sum of Disjoint Products approach consists in finding a combinatorial decomposition of the sets of monomials in $I_j(S)$ into disjoint sets (see Figure \ref{fig:Pbasis}). This can be done in several ways, e.g. Rees and Stanley decompositions \cite{R56,S78}.  A computationally efficient approach to these decompositions uses the concept of {\em involutive basis} of monomial ideals. Since this is not a widely known concept, let us introduce it here.
Involutive bases were introduced in \cite{GB98a,GB98b} and an extensive study of their role in commutative algebra is given in \cite{S09a,S09b,S10}. 
They are a type of Gr\"obner bases with additional combinatorial properties.
%A survey of involutive divisions and their role in commutative algebra and the algebraic approach to partial differential equations can be seen in \cite{S09a,S09b,S10} where the particular case of Pommaret bases is studied deeply.

For any subset $N\subseteq\{1,\dots,n\}$, we denote $\mathbb{N}^n_N=\{\nu \in \mathbb{N}^n_0\vert \forall j\notin N, n_j=0\}$. The only non-zero entries of the multi-indices in $\mathbb{N}^n_N$ occur at the positions given by $N$.

\begin{definition}\label{def:division}
Let $\mathcal{N}\subset \mathbb{N}^n$ be a finite set, and $L$ an assignment of a subset $N_{L,\mathcal{N}}(\nu)\subseteq\{1,\dots,n\}$ of indices to every multi-index $\nu\in\mathcal{N}$. We say that $L$ is an {\em involutive division} if the involutive cones  $\mathcal{C}_{L,\mathcal{N}}(\nu)=\nu+\mathbb{N}^n_{N_{L,\mathcal{N}}(\nu)}$ satisfy that:
\begin{enumerate}
\item If there exist  $\mu,\nu\in\mathcal{N}$, $\mu\neq\nu$, such that $\mathcal{C}_{L,\mathcal{N}}(\mu)\cap \mathcal{C}_{L,\mathcal{N}}(\nu)\neq \emptyset$, then $\mathcal{C}_{L,\mathcal{N}}(\mu)\subseteq\mathcal{C}_{L,\mathcal{N}}(\nu)$ or $\mathcal{C}_{L,\mathcal{N}}(\nu)\subseteq \mathcal{C}_{L,\mathcal{N}}(\mu)$, i.e. there are no non-trivial intersections between involutive cones. 
\item If $\mathcal{N}'\subset \mathcal{N}$ then $N_{L,\mathcal{N}}(\nu)\subseteq N_{L,\mathcal{N}'}(\nu)$ for all $\nu\in\mathcal{N}'$.
\end{enumerate}
If $L$ is an involutive division, we say that the elements of  $N_{L,\mathcal{N}}(\nu)\subseteq\{1,\dots,n\}$ are the multiplicative indices of $\nu$.
\end{definition}

If $i$ is a multiplicative index, we say that $x_i$ is a multiplicative variable. The set of multiplicative indices (or variables) of a multi-index $\nu$ with respect to the involutive division $L$ and a set $\mathcal{N}$ is denoted by $\Xk_{L,\mathcal{N}}(\nu)$, and the set of non-multiplicative indices (or variables) is denoted by $\overline{\Xk}_{L,\mathcal{N}}(\nu)$. We say that $x^\mu$ is an {\em involutive divisor} of $x^\nu$ with respect to $L$ if $x^\mu\vert x^\nu$ and $x^{\nu-\mu}\in \kb[\Xk_{L,\mathcal{N}}({\mu})]$.

The following are the two main examples of involutive divisions.
\begin{definition}[Janet division]
Consider the following subsets of the given set $\mathcal{N}\subset\mathbb{N}^n_0$:
\[
(d_1,\dots,d_k)=\{\nu\in\mathcal{N} \vert \nu_i=d_i,1\leq i \leq k\}.
\]
The index $1$ is Janet-multiplicative for $\nu\in\mathcal{N}$ if $\nu_1=max_{\mu\in\mathcal{N}}\{\mu_1\}$. Any index $1<k$ is multiplicative for $\nu\in(d_{1},\dots,d_{k-1})$ if $\nu_k=\max_{\mu\in(d_{1},\dots,d_{k-1})}\{\mu_k\}$.
\end{definition}

\begin{definition}[Pommaret division]
Let $\mu=(\mu_1,\dots,\mu_n)\in \mathbb{N}^n$. We say that the {\em class} of $\mu$ or $x^\mu$, denoted by $\cls(\mu)=\cls(x^\mu)$, is equal to $\max\{ i\vert \mu_i\neq 0\}$. The {\em multiplicative variables} of $x^\mu$ with respect to the Pommaret division, are $\Xk_P(\mu)=\Xk_P(x^\mu)=\{x_{\cls (\mu)},\dots,x_n\}$.
\end{definition}

The assignment of multiplicative and non-multiplicative variables is independent of the set $\mathcal{N}$ in the case of the Pommaret division, in such a case we say that the involutive division is {\em global}. The Janet division is not global.

\begin{definition}
A finite collection of monomials $\Bc\subseteq R$ is an {\em involutive basis} of the monomial ideal $I=\langle \Bc \rangle$ with respect to the involutive division $L$ if $I=\bigoplus_{h\in\Bc} h\cdot\kb[\Xk_{L,\Bc}(h)]$ as vector spaces.
%A finite polynomial set $\Bc$ is a Pommaret basis of the polynomial ideal $I=\langle \Bc\rangle$ for the term order $\prec$, if all the elements of $\Bc$ possess distinct leading terms and these terms form a Pommaret basis of the leading ideal $\lt (I)$.
\end{definition}

If every finite set of monomials possesses a finite involutive basis with respect to a certain involutive division $L$, we say that $L$ is {\em Noetherian}. The Janet division is Noetherian, but the Pommaret division is not, see for instance the ideal $I=\langle xy\rangle\subseteq\kb[x,y]$. Those monomial ideals which do possess a finite Pommaret basis are called {\em quasi-stable} ideals \cite{S09b}.

\begin{example}\label{ex:quasi-stable-ideal}
Let $I=\langle x_1^2, x_2^3\rangle\subseteq\kb[x_1,x_2]$. Let $P$ be the Pommaret division. Then $\Xk_P(x_1^2)=\{x_1,x_2\}$ and $\Xk_P(x_2^3)=\{x_2\}$. The monomial $x_1x_2^3$ is in $I$ but it is not in the involutive cone of any of its minimal generators, hence the minimal generating set of $I$ is not a Pommaret basis of it. 
The involutive basis of $I$ with respect to the Pommaret division is given by $\Bc=\{x_1^2,x_1x_2^3,x_2^3\}$, therefore $I$ is a quasi-stable ideal. Figure \ref{fig:Pbasis} shows on one side the minimal generating set  of $I$ and their usual (overlapping) multiplicative cone of each of its elements, and on the other side the Pommaret basis of $I$ and the involutive (non-overlapping) cone of each of its elements. For this ideal, the Pommaret and Janet basis with respect to $\Bc$ coincide.

\begin{figure}[t]
\begin{center}
\begin{tikzpicture}[scale=0.63, transform shape]
\tikzset{dot/.style={draw=black,thick,fill=white,circle}}
\tikzset{wdot/.style={draw=black,fill=white,circle}}

\fill[red!20!white] (0,3) rectangle (5.7,5.7);
\fill[red!20!white] (2,0) rectangle (5.7,5.7);
\fill[red!40!white] (2,3) rectangle (5.7,5.7);

%\fill[red!20!white] (2,.7) rectangle (5.7,1.3);
%\fill[red!20!white] (2,1.7) rectangle (5.7,2.3);

\foreach\l[count=\y] in {$x_2$, $x_2^2$,$x_2^3$, $x_2^4$,$x_2^5$}
{
\draw (0,\y) -- (5.5,\y);
}

\draw(0,0)--(6,0);
\draw(0,0)--(0,6);
\foreach \l[count=\x] in {$x_1$, $x_1^2$,$x_1^3$, $x_1^4$,$x_1^5$}
{
\draw (\x,0) -- (\x,5.5);
}
\node at (1.5,-0.3){$x_1^2$};
\node at (-0.5,2.7){$x_2^3$};

\node[dot] at (0,3){};
\node[dot] at (2,0){};
\fill(0,3) circle (.1cm);
\fill(2,0) circle (.1cm);

%\node[wdot] at (2,1){};
%\node[wdot] at (2,2){};
%\node at (1.5,.7){$x_1^2x_2$};
%\node at (1.5,1.7){$x_1^2x_2^2$};

\foreach \i in {2,...,5}{
    \foreach \j in {0,...,2}{
    {
     \fill(\i,\j) circle (.1cm);
     }
    }
   }
   
 \foreach \i in {0,...,5}{
    \foreach \j in {3,...,5}{
    {
     \fill(\i,\j) circle (.1cm);
     }
    }
   }

\fill[red!20!white] (10,0) rectangle (13.7,5.7);
\fill[red!20!white] (7.7,3) rectangle (8.3,5.7);
\fill[red!20!white] (8.7,3) rectangle (9.3,5.7);

\foreach\l[count=\y] in {$x_2$, $x_2^2$,$x_2^3$, $x_2^4$,$x_2^5$}
{
\draw (8,\y) -- (13.5,\y);
}

\draw(8,0)--(14,0);
\draw(8,0)--(8,6);
\foreach \l[count=\x] in {$x_1$, $x_1^2$,$x_1^3$, $x_1^4$,$x_1^5$}
{
\draw (\x+8,0) -- (\x+8,5.5);
}
\node at (9.5,-0.3){$x_1^2$};
\node at (7.5,2.7){$x_2^3$};

\node[dot] at (8,3){};
\node[dot] at (10,0){};
\fill(8,3) circle (.1cm);
\fill(10,0) circle (.1cm);

\fill(10,1) circle(.1cm);
\fill(10,2) circle(.1cm);

\foreach \i in {11,...,13}{
    \foreach \j in {0,...,2}{
    {
     \fill(\i,\j) circle (.1cm);
     }
    }
   }
   
 \foreach \i in {8,...,13}{
    \foreach \j in {3,...,5}{
    {
     \fill(\i,\j) circle (.1cm);
     }
    }
   }
   
   \node[wdot] at (9,3){};
\node at (9,2.6){$x_1x_2^3$};
   
\end{tikzpicture}
\caption{Minimal generating set (left) and Pommaret basis (right) for the ideal $I=\langle x_1^2, x_2^3\rangle$. Note that the usual cones of each of the generators overlap on all their common multiples (left), and that the involutive cones (right) do not overlap, hence obtaining a partition of the set of monomials in $I$ into disjoint sets.}\label{fig:Pbasis}
\end{center}
\end{figure}
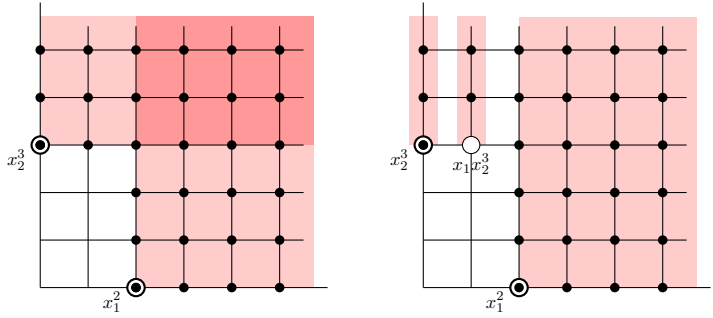

\end{example}

%In the case that a finite involutive basis exists there is a well defined algorithm to obtain it \cite{S09a}. Quasi-stable monomial ideals have appeared in the literature under the names of {\em ideals of nested type} \cite{BG06}, {\em weakly stable ideals} \cite{CS05} or {\em ideals of Borel-type} \cite{HPV03}.
%The name quasi-stable is due to the fact that these ideals are a generalization of {\em stable ideals} in the sense that a monomial ideal $I$ is stable if and only if its minimal monomial generating set is also a Pommaret basis for $I$.

If we have a finite involutive basis $\Bc$ for a monomial ideal $I$ and an involutive division $L$, Definition \ref{def:division} shows that we directly obtain a disjoint partition of the set of monomials in $I$:
\begin{equation}\label{eq:partition}
mon(I)=\bigoplus_{h\in\Bc} h\cdot\kb[\Xk_{L,\Bc}(h)].
\end{equation}

Let $x^\mu\in I_{S,l}$ be a monomial, and $\Xk_{L,\mathcal{N}}(x^\mu)$ be its set of multiplicative variables with respect to the involutive division $L$ and the set $\mathcal{N}$; let $\overline{\Xk}_{L,\mathcal{N}}(x^\mu)$ be its set of non-multiplicative variables. We denote the probability of the involutive cone of $x^\mu$ with respect to $L$ and $\mathcal{N}$ by $\widehat{pr}_L(x^\mu,\mathcal{N})=pr(\bigwedge_{i\in{\Xk}_{L,\mathcal{N}}(x^\mu)}(c_i\geq \mu_i) \bigwedge_{i\in \overline{\Xk}_{L,\mathcal{N}}(x^\mu)}(c_i=\mu_i))$ which in the case of independent components, can be computed as $\prod_{i\in\Xk_{L,\mathcal{N}}(x^\mu)}p_{i,\mu_i}\prod_{i\in\overline{\Xk}_{L,\mathcal{N}}(x^\mu)}\widehat{p}_{i,\mu_i}$, where $\widehat{p}_{i,\mu_i}=pr(c_i=\mu_i)$. Since we consider the set $\mathcal{N}$ to be the involutive basis  $\Bc$ of the ideal $I_{S,l}$, we can drop $\mathcal{N}$ from the notation.

\begin{proposition}\label{prop:algebraicSDP}
Let $\Bc$ be an involutive basis of the $l$-reliability ideal $I_{S,l}$ with respect to an involutive division $L$. The $l$-reliability of system $S$ is given by
 \begin{equation}\label{eq:reliability-involutive}
     R_l(S)=pr(I_{S,l})=\sum_{h\in\Bc}\widehat{pr}_L(h).
   \end{equation}  
\end{proposition}

\begin{proof}
The first equality is given by the algebraic description of the system's reliability. For the second one, consider the disjoint decomposition (\ref{eq:partition}) of $I_{S,l}$ given by $\Bc$. The set of monomials in $I_{S,l}$ is the disjoint union of the involutive cones of the elements on $\Bc$. The probability associated to the involutive cone of $h\in\Bc$ is given by $\widehat{pr}_L(h)$ and since the union of these cones is disjoint, the probability of the union equals the sum of the probabilities of cones, as claimed.
\end{proof}
   
Proposition \ref{prop:algebraicSDP} is an algebraic version of the sum-of-disjoint products approach to the evaluation of the system's reliability \cite{KZ03}. 
   
 \begin{example}
 Consider the ideal from Example~\ref{ex:bridge}.
 %The minimal paths of this binary system are $MP_1=x_1x_2$, $MP_2=x_4x_5$, $MP_3=x_1x_3x_5$ and $MP_4=x_2x_3x_4$., hence its reliability ideal is 
 %\[
 %I(S)=\langle x_1x_2,x_4x_5,x_1x_3x_5,x_2x_3x_4\rangle,
 %\]
% which is not squarefree stable or strongly stable with respect to any ordering of the variables. The minimal free resolution of $I(S)$ gives the following expression for the reliability of $S$ from the numerator of the Hilbert series of $I(S)$:
 %\begin{align*}
 %R(S)=&p_1p_2{+}p_4p_5{+}p_1p_3p_5{+}p_2p_3p_4{-}(p_1p_2p_3p_4{+}p_1p_2p_3p_5{+}p_1p_2p_4p_5{+}p_1p_3p_4p_5{+}p_2p_3p_4p_5)\\
 %&+2 p_1p_2p_3p_4p_5.
 %\end{align*}
 The Janet basis for the reliability ideal of the bridge structure is 
 \[
 \Bc=\{x_1x_2,x_4x_5, x_1x_3x_5,x_2x_3x_4,x_1x_4x_5,x_2x_4x_5 \}.
 \]
 The Janet non-multiplicative variables for the elements of $\Bc$ are: $\bar{\Xk}(x_1x_2)=\emptyset$,  $\bar{\Xk}(x_4x_5)=\{x_1,x_2\}$, $\bar{\Xk}(x_1x_3x_5)=\{x_2\}$, $\bar{\Xk}(x_2x_3x_4)=\{x_1\}$, $\bar{\Xk}(x_1x_4x_5)=\{x_2,x_3 \}$ and $\bar{\Xk}(x_2x_4x_5)=\{x_1,x_3\}$. Using (\ref{eq:reliability-involutive}) we obtain the reliability of this system:
 \begin{align*}
 R(S)=&p_1p_2+\widehat{p}_1\widehat{p}_2p_4p_5+p_1\widehat{p}_2p_3p_5+\widehat{p}_1p_2p_3p_4+p_1\widehat{p}_2\widehat{p}_3p_4p_5\\
 &+\widehat{p}_1p_2\widehat{p}_3p_4p_5.
 \end{align*}
 \end{example}

\section{Algebraic algorithms for stable systems}\label{sec:algebraicStable}
The two algebraic approaches described in Section \ref{sec:algebraic} are particularly efficient in the case of stable  and strongly stable systems, both binary and multi-state. On the one hand, we have that closed form formulas are known for the minimal free resolutions of the ideals corresponding to stable and strongly stable systems. This means that we can obtain IIE formulas in an efficient way using these resolutions. On the other hand, involutive bases for the ideals corresponding to stable and strongly stable bases are small and easy to obtain, hence the algebraic version of the SDP method is particularly efficient.

\subsection{Stable and strongly stable ideals}
\begin{definition}
Let $\R=\kb[x_1,\dots,x_n]$, a monomial ideal $I\subset \R$ is called {\em strongly stable} if for any monomial $m\in I$ we have that the monomial $\frac{x_jm}{x_i}$ is in $I$ for every $j<i$. $I$ is called {\em stable} if  $\frac{x_jm}{x_{\max(m)}}$ is in $I$ for every $j<\max(m)$, where $\max(m)$ is the biggest index of a variable dividing $m$. 
\end{definition}

In the binary case, the ideal corresponding to the system is square-free, and the above definitions need adaptation, as the only stable and square-free ideal is generated by the variables themselves.

%In the binary case, the ideal corresponding to the system is squarefree and the above definitions need to be adapted, for the only stable {\em and} squarefree ideal is the ideal generated by the variables themselves.
\begin{definition}
A squarefree monomial ideal $I\subset \R$ is called {\em squarefree strongly stable} if for any monomial $m\in I$ we have that the monomial $\frac{x_jm}{x_i}$ is in $I$ for every $j<i$ such that $x_j$ does not divide $m$. $I$ is called {\em squarefree stable} if  $\frac{x_jm}{x_{\max(m)}}$ is in $I$ for every $j<\max(m)$ such that $x_j$ does not divide~$m$.
\end{definition}

Given a monomial ideal $I$, we say that the stable (resp. strongly stable) closure of $I$ is the smallest stable (resp. strongly stable) ideal $\overline{I}$ such that $I\subseteq\overline{I}$.

One way to encode the monomials of an ideal, which in the case of reliability ideals encode the states of a system, is by using {\em cumulative exponents}, see \cite{DFMSS19}.
\begin{definition}
The {\em cumulative exponent} of a monomial  $m=x_1^{a_1}x_2^{a_2}\cdots x_n^{a_n}$ is defined as
\[
\sigma(m)=(\sigma_1(m),\sigma_2(m),\dots,\sigma_n(m)),
\]
where $\sigma_i(m)=a_i+a_{i+1}+\cdots+a_n$.
\end{definition}

 It is easy to see that $\sigma_1(m)$ equals the total degree of $m$ and that the cumulative exponent of any monomial is a monotone non-increasing sequence. We can also obtain the monomial corresponding to such a sequence for a non-increasing vector  $\sigma$, the corresponding monomial $m$ is given by 
\[
m=x_1^{\sigma_1-\sigma_2}\cdots x_{n-1}^{\sigma_{n-1}-\sigma_n}x_n^{\sigma_n}.
\]

\begin{proposition}Let $S$ be a binary coherent system. Algorithm \ref{alg:closure} computes the strongly stable closure of $S$ by means of its correspondent reliability ideal. 
\end{proposition}

\begin{proof}
Let $I_S$ be the reliability ideal of the system $S$ and $m\in I_S$ a monomial. If $m'$ is in the strongly stable closure of $\langle m\rangle$ (the ideal generated by the monomial $m$) then $\sigma_i(m')\leq\sigma_i(m)$ for all $i$ and the inequality is strict for some $i$. In particular, let $m$ be a monomial such that $x_i$ divides $m$. Then the cumulative exponent of $m'=\frac{x_jm}{x_i}$ for $j<i$ is given by $\sigma_k(m')=\sigma_k(m)$ for $k\leq j$ and $k>i$, and $\sigma_k(m')=\sigma_k(m)-1$ for $j+1\leq k\leq i$. We use these observations to build the main loop of the algorithm, in which we consider all possible monomials to be included in the strongly stable closure of $I_S$.

With respect to termination, observe that in lines $7$ and $11$ of the algorithm, $g(\sigma)$ and $g(c)$ denote the monomial corresponding to the cumulative vectors $\sigma$ and $c$. Observe that in each step of the main loop (lines $4$ to $16$) we extract one element from $P$ in line $5$ but in the loop in lines $9$ to $14$ we (possibly)  introduce several elements in $P$. The termination of the algorithm is however ensured by the fact that the elements introduced in line $12$ are strictly smaller than the element extracted in line $6$, hence by a good ordering argument, we eventually extract all the elements in $P$.
\end{proof}

\begin{remark}
Algorithm \ref{alg:closure} can be used to obtain the stable closure of the input by considering in line $9$ only the last nonzero exponent of $g(\sigma)$. Moreover, for squarefree ideals the algorithm can be easily modified, for instance adding in line $11$ the condition that $g(c)$ is squarefree.
\end{remark}

\begin{algorithm}
\caption{Strongly stable closure of monomial ideal}\label{alg:closure}
\begin{algorithmic}[1]
\Require Set of monomials $\{g_1,\dots,g_r\}\in\kb[x_1,\dots,x_n]$
\Ensure Set of generators of the strongly stable closure of $I=\langle g_1,\dots,g_r\rangle$
\State $M \gets \{g_1,\dots,g_r\}$
\State $P \gets \{\sigma(g_1),\dots,\sigma(g_r)\}$
\State $D \gets \emptyset$
\While{$P \neq \emptyset$}
\State $\sigma \gets$ Last element in $P$ by lexicographic order
\State $P \to \sigma$
\If{$g(\sigma)$ is not divisible by any element in $D$} 
    \State $D \gets g(\sigma)$
        \ForAll{$i\in\{2,\dots,n\}$ such that $g_i>0$ }
        \State $c\gets \sigma-(0,\dots,\stackrel{i}{1},\dots,0)$
        \If{$c$ is decreasing {\bf and} $g(c)$ is not divisible by any element in $D$} 
         \State $P\gets c$
        \EndIf
        \EndFor
\EndIf
\EndWhile
\State \Return $D$
\end{algorithmic}
\end{algorithm}

\subsection{Free resolutions of reliability ideals of stable systems}\label{sec:resolutionStable}
\subsubsection{Multi-state systems}
Let $S$ be a system with $n$ multi-state components. If $S$ is a (strongly) stable system for level $l$, then its corresponding ideal $I_{S,l}$ is a (strongly) stable ideal. The resolution described in \cite{EK90} by Eliahou and Kervaire is an explicit form of the minimal free resolution for stable ideals, hence it is also valid for strongly stable ideals. To describe this resolution we need the following notation: we call {\em admissible symbol} any pair $[m;u]$ where $m$ is a monomial in the minimal monomial generating set of $I_{S,l}$ and $u$ is an increasing set of variables such that $\max(u)<\max(m)$. Then, if $I$ is stable, there is a generator of the $i$-th module of the minimal free resolution of $I$ for each admissible symbol with $\vert u\vert=i$. We say that the multi-degree of an admissible symbol is given by $\md([m;u])=\md(m\cdot x_{u_1}\cdots x_{u_i})$ where $u=(u_1<\cdots<u_i)$. 

\begin{proposition}
The $l$-reliability of a stable system $S$ with multi-state components is given by 
\[
R_l(S)=\sum_{i=0}^d (-1)^{i}\sum_{\vert u\vert=i}pr(\md([m;u])),
\]
where the inner sum runs through all admissible symbols and $d$ is the maximal length of any sequence $u$ such that $[m;u]$ is an admissible symbol. 
\end{proposition}

\begin{proof}
The $l$-reliability ideal of $S$ is stable. The minimal free resolutions of stable ideals are given in \cite{EK90}, which are supported on the admissible symbols of $I_l$. Then, by equation (\ref{eq:HilbertFormula}) we have the claimed form of the $l$-reliability of the system, $R_l(S)$.
\end{proof}

\begin{example}\label{ex:multi-state-stable}
Let $S$ be a multi-state system with three components such that the ideal corresponding to level $l=2$ is 
\[
I_{2,S}=\langle x_1^2,x_1x_2,x_2^2,x_1x_3^2 \rangle,
\]
which is a stable ideal. The list of admissible symbols for $I_{2,S}$ is 
\begin{align*}
&\{[x_1^2;\emptyset], [x_1x_2;\emptyset],[x_2^2;\emptyset],[x_1x_3^2;\emptyset],\\
&[x_1x_2;\{1\}],[x_2^2;\{1\}],[x_1x_3^2;\{1\}],[x_1x_3^2;\{2\}],[x_1x_3^2;\{1,2\}]\}
\end{align*}
and hence the numerator of the Hilbert series of $I_2(S)$ is
\begin{align*}
HN(I_{2,S})&=x_1^2+x_1x_2+x_2^2+x_1x_3^2\\
&-(x_1^2x_2+x_1x_2^2+x_1^2x_3^2+x_1x_2x_3^2)+x_1^2x_2x_3^2.
\end{align*}
Assuming the probabilities $p_{i,j}=pr(x_i\geq j)$ of component $i$ to be working at level at least $j$ are given~by
\begin{align*}
&p_{1,1}=0.9,\, p_{1,2}=0.8,\, p_{2,1}=0.85,\, \\
&p_{2,2}=0.8,\, p_{3,1}=0.75,\, p_{3,2}=0.7,
\end{align*}
then the $2$-reliability of the system, i.e. the probability that the system is operating at level at least $2$, is
\begin{align*}
R_{2}(S)=&0.8{+}0.9{\cdot}0.85{+}0.8{+}0.9{\cdot}0.7{-}(0.8{\cdot}0.85\\
+&0.9{\cdot}0.8{+}0.8{\cdot}0.7{+}0.9{\cdot}0.85{\cdot}0.7){+}0.8{\cdot}0.85{\cdot}0.7\\
=&0.9755.
\end{align*}
\end{example}

\begin{remark}
Stable and strongly stable ideals are very important objects in commutative algebra and have been object of intense research. Besides the seminal paper \cite{EK90} that explicitly describes the minimal free resolution of stable ideals, we refer the reader to \cite{FMS11} for a deeper study of these ideals.
\end{remark}

\subsubsection{Binary systems}
A resolution of the type described above for squarefree stable ideals was given in \cite{AHH97,AHH98}. In the squarefree case the admissible symbols for an ideal $I$ are those $[m,u]$ such that $m$ is a minimal monomial generator and $u$ is a sequence $u$ is an increasing set of variables such that $\max(u)<\max(m)$ and no $u_i$ in $u$ divides $m$. The minimal free resolution of $I$ is then supported on the admissible symbols for $I$.

\begin{example}
Let $S$ be a binary system with four components whose reliability ideal is 
\[
I_S=\langle x_1x_2,x_1x_3,x_1x_4,x_2x_3\rangle.
\]
The ideal $I_S$ is squarefree strongly stable with  admissible symbols:
\begin{align*}
&[x_1x_2;\emptyset], [x_1x_3;\emptyset],[x_1x_4;\emptyset],[x_2x_3;\emptyset],[x_1x_3;\{2\}],[x_1x_4;\{2\}],\\
&[x_1x_4;\{3\}],[x_2x_3;\{1\}],[x_1x_4;\{2,3\}].
\end{align*}
Hence, the reliability of the system is given by
\begin{align*}
R(S)=&p_1p_2+p_1p_3+p_1p_4+p_2p_3\\
-&(2p_1p_2p_3+p_1p_2p_4+p_1p_3p_4)+p_1p_2p_3p_4,
\end{align*}
where $p_i$ is the working probability of component $i$.

\end{example}

\subsection{Multiplicative variables and involutive bases for stable and squarefree stable ideals.}

For any involutive division that has the property of being constructive (this is a technical requirement that both the Pommaret and Janet divisions satisfy), Seiler gives in \cite{S09a} a completion algorithm, which given a generating set of $I$, produces an involutive basis of the ideal. In the case of the Pommaret division and monomial ideals, we need the ideal $I$ to be quasi-stable for the algorithm to terminate. In this case, the set of Janet-multiplicative variables and Pommaret-multiplicative variables coincide for any monomial $\xb^\mu$ in the involutive basis. In case the ideal $I$ is not quasi-stable, then the set of Pommaret-multiplicative variables is always included in the set of Janet-multiplicative variables for every monomial in the involutive basis. This inclusion is strict for some monomials.

A relevant result for stable ideals, and the one that justifies the name {\em quasi-stable} for ideals possessing a finite Pommaret basis, is the following.

\begin{proposition}[\cite{S09b}, Proposition 8.6] \label{prop:stableIdeal}
A monomial ideal $I$ is stable if and only if its minimal monomial generating set is also a Pommaret basis for $I$.
\end{proposition}

Therefore, for stable and strongly stable multi-state systems, the computation of the reliability of the system using the algebraic version of the Sum of Disjoint Product method described in Section \ref{sec:algebraic} is straightforward. Since the reliability ideals of these systems are stable, Proposition \ref{prop:stableIdeal} tells us that all we need is to compute the sets of multiplicative and non-multiplicative variables to obtain the reliability of the system. Since for these ideals the Janet and Pommaret multiplicative variables coincide, we use the simpler one to compute, namely the Pommaret multiplicative variables.

\begin{example}
The ideal $I_{2,S}$ in Example \ref{ex:multi-state-stable} is a stable ideal, hence its minimal generating set is itself a Pommaret basis. The  Pommaret multiplicative variables for any monomial $m$ are easy to compute, for they are the set $\{x_{\max(m)},\dots,x_n\}$. The sets of non-multiplicative variables for the generators of this ideal are
\begin{align*}
&\bar{\Xk}(x_1^2)=\emptyset,\, \bar{\Xk}(x_1x_2)=\{x_1\},\\
&\,\bar{\Xk}(x_2^2)=\{x_1\},\,\bar{\Xk}(x_1x_3^2)=\{x_1,x_2\}.
\end{align*}
Hence the $2$-reliability of this system is given by
\begin{align*}
R_2(S)=&p_{1,2}+\widehat{p}_{1,1}p_{2,1}+\widehat{p}_{1,0}p_{2,2}+\widehat{p}_{1,1}\widehat{p}_{2,0}p_{3,2}\\
=&0.8+0.1\cdot0.85+0.1\cdot0.8+0.1\cdot0.15\cdot0.7\\
=&0.9755,
\end{align*}
which is the same result obtained with the IIE method in Example \ref{ex:multi-state-stable}.
\end{example}

In the squarefree case, i.e. for binary systems, the situation seems a bit more difficult, for these ideals are almost never quasi-stable. In this case, we need to use Janet bases. We have, however, a result analogous to Proposition \ref{prop:stableIdeal}. We provide here the proof of this result and refer to \cite{IS24} for more details on monomial ideals whose minimal generating set is a Janet basis.

\begin{proposition}[\cite{IS24}, Theorem 3.2]
Let $I\subseteq R$ be a squarefree stable monomial ideal, then its minimal generating set is a Janet basis for $I$. 
\end{proposition}
\begin{proof}
    Let $I$ be a square-free stable ideal and $G(I)$ its minimal generating set. Since the Janet division is continuous and constructive, then it suffices to show that for any $h\in G(I)$ and $x_i\in\overline{\mathcal{X}}_{J,G(I)}(h)$ there is a Janet involutive divisor of $x_i h$ in $G(I)$.

    Since $I$ is quasi-stable, we have that $h'=x_i\frac{h}{x_{\min(h)}}$ is in $I$, and $h'$ is an involutive divisor of $x_ih$, since $\frac{x_ih}{h'}=\min(h)$, which is smaller or equal than $\min(h')$ and hence multiplicative for it. Hence we do no need to add $x_ih$ to complete $G(I)$ to an involutive basis, and we are done.
\end{proof}

The computation of Janet multiplicative variables is not as straightforward as for the Pommaret division. There are, however, efficient algorithms for their computation, based on the so-called Janet tree structure, cf. \cite{GBY01,S10,C19}. Once the sets of multiplicative variables of the generators of $I_S$ are computed, we can directly obtain the reliability of the system $S$.

\subsection{Algebraic importance measures for stable systems}\label{sec:algebraicImportance}
In \cite{SW15}, an algebraic alternative to structural importance is given, based on the Hilbert function of its reliability ideal. Let $S$ be a system with $n$ components such that each component can be in states $\{0,1,\dots,M_i\}$, and $I_{S,l}$ its $l$-reliability ideal (in this Section we denote it by $I_l$ if the system $S$ is clear from the context). Let $\widehat{I}_l$ be the Artinian closure of $I_l$, i.e. $\widehat{I}_l=I_l+\langle x_1^{M_1+1}, \dots, x_n^{M_n+1}\rangle$, in the general case -i.e. for any monomial ideal $I$ not necessarily coming from a coherent system-, to obtain $\widehat {I}$ the exponents $M_i$ are given by the highest exponent to which $x_i$ is raised in any minimal generator of $I$. The ideal $\widehat{I}_l$ is a zero-dimensional ideal, which means that the number of monomials not in $\widehat{I}_l$ is finite. For any zero-dimensional ideal, the number of monomials not in it is called the {\em multiplicity} of the ideal. For instance, the multiplicity of the ideal $\langle x_1^{M_1+1}, \dots, x_n^{M_n+1}\rangle$ is equal to $N=\prod_{i=1}^n {M_i+1}$. In the context of reliability ideals of coherent systems, we define the {\em algebraic multiplicity} (or simply multiplicity) of component $c_i$, denoted by $\mult(c_i)$ as the multiplicity of the ideal $\widehat{I}_{l,\overline{i}}$ which is generated by the monomials $\{\frac{\mu}{x_i^\infty}\mbox{ s.th. }\mu \mbox{ is a monomial generator of } \widehat{I}_l\}$ where $\frac{\mu}{x_i^\infty}$ means that we have deleted variable $i$ from $\mu$.

\begin{definition}
The {\em multiplicity importance} for level $l$ of component $c_i$ of system $S$ is the number $imp_\mult(c_i)=N-\mult(\widehat{I}_{l,\overline{i}})$.
\end{definition}

Observe that the multiplicity importance for level $l$ of a component is inversely proportional to the multiplicity of its associated ideal $\widehat{I}_{l,\overline{i}}$. In the case of binary systems, the multiplicity importance (for level $1$) is equivalent to the structural importance, see \cite{SW15}. In the case of multi-state systems we have a measure of multiplicity importance for each level $l$. The interpretation of the importance of each component is then more subtle, since a component $c_i$ could have higher importance than another component $c_j$ for certain levels and the situation can be the opposite for other levels. 
%In order to have a description of the situation one could encode the importance of the components by a bar code of multiplicity importance.

In the case of strongly stable systems, the ordering of the variables is equivalent to the ordering based on multiplicity importance.

\begin{theorem}
Let $I\subseteq \R=\kb[x_1,\dots,x_n]$ be a strongly stable squarefree ideal. If $i<j$ then $\mult(\widehat{I}_{\overline{i}})\leq\mult(\widehat{I}_{\overline{j}})$.
\end{theorem}
\begin{proof}
Without loss of generality, we can consider that the ideal $I$ is equi-generated, i.e. all generators are of the same degree, say $d$. It is enough to check the minimal generators of $I$. For each generator $m$ of $I$ we can be in one of these situations:
\begin{enumerate}
    \item $x_i$ does not divide $m$ but $x_j$ does,
    \item both $x_i$ and $x_j$ divide $m$,
    \item $x_i$ divides $m$ but $x_j$ does not,
    \item none of $x_i$ and $x_j$ divide $m$.
\end{enumerate}

Observe that since $I$ is strongly stable, for each generator in situation $(1)$, there is another generator in situation $(3)$, therefore it is enough to observe situations $(2)$, $(3)$ and $(4)$.

If $m$ is in situation $(2)$ then we have that $m/x_i\in \widehat{I}_{\overline{i}}$ and $m/x_j\in \widehat{I}_{\overline{j}}$, both elements are of degree $d-1$ and they do not divide each other, hence, every $m$ of type $(2)$ contributes with one degree $d-1$ element to both $\widehat{I}_{\overline{i}}$ and $\widehat{I}_{\overline{j}}$.

If $m$ is in situation $(3)$ then $m/x_i\in \widehat{I}_{\overline{i}}$ and $m\in \widehat{I}_{\overline{j}}$ and since $m/x_i$ divides $m$, generators of type $(3)$ contribute to more generators of smaller degree to $\widehat{I}_{\overline{i}}$ than to $\widehat{I}_{\overline{j}}$.

Finally, if $m$ is of type $(4)$, then $m$ is in both $\widehat{I}_{\overline{i}}$ and $\widehat{I}_{\overline{j}}$.

Putting together the information of these four types of generators, we have that $\widehat{I}_{\overline{i}}$ has at least as many generators of degree $d-1$ as $\widehat{I}_{\overline{j}}$, and hence, its multiplicity is smaller than or equal than that of $\widehat{I}_{\overline{j}}$.
\end{proof}

\begin{example}
Let us consider the ideal $I=\langle ab,ac,bc,cd\rangle$ from Example \ref{ex:fourcomponents}. It is squarefree strongly stable with respect to the order $c<a<b<d$.
The ideals corresponding to deletion of each of the variables are the following:
\[
\widehat{I}_{\overline{a}}=\langle b,c,d^2\rangle,\, \widehat{I}_{\overline{b}}=\langle a,c,d^2\rangle,\, \widehat{I}_{\overline{c}}=\langle a,b,d\rangle,\, \widehat{I}_{\overline{d}}=\langle a^2,b^2,ab,c\rangle
\]
Their multiplicities are: 
\[
\mult(\widehat{I}_{\overline{a}})=2,\,\mult(\widehat{I}_{\overline{b}})=2,\,\mult(\widehat{I}_{\overline{c}})=1,\,\mult(\widehat{I}_{\overline{d}})=3.
\]
Hence, the ordering by multiplicity importance of the components corresponding to the variables of this ideal is $(c,a,b,d)$ or $(c,b,a,d)$, which are the orderings for which $I$ is strongly stable.
\end{example}

\section{Computer experiments and examples}\label{sec:experiments}
The problem of computing system reliability is NP-hard, therefore the implementation of efficient algorithms is key to obtain good results in actual applications. There is a great variety of algorithms for system reliability computation, see for instance \cite{TB17}, which are generally divided into two categories. On the one hand,  there are general algorithms, making use of mathematical concepts or efficient structures to encode the systems' states in order to avoid as much redundancy as possible. These include Binary Decision Diagrams, Sum of Disjoint Products or Universal Generating Functions among others. The other main approach is to construct specific algorithms for particular classes of systems. Some of the systems more frequently studied in this respect are series-parallel systems, $k$-out-of-$n$ systems both binary and multi-state and its variants, or networks, among others.

The algebraic methodology for reliability computation using monomial ideals falls into both of these two categories. On the one hand, it gives algebraic versions of general approaches: compact Inclusion-Exclusion formulas, and Sum of Disjoint Products. On the other hand, the method can be adapted to particular systems providing efficient specific algorithms \cite{SW10,MSW18,PSW21,PS21}. The algebraic approach is based, like others, on avoiding as much redundancy as possible when enumerating the states needed for the final reliability computation. In the case of compact Inclusion-Exclusion formulas, this is provided by the possibility of using different resolutions to express the numerator of the Hilbert series of the system's ideals. In this respect, a fast computation of the minimal resolution or close-to-minimal resolutions is of paramount importance to our approach. This methodology can be approached as a recursive procedure, computing the Hilbert series of an ideal in terms of the Hilbert series of smaller ideals. Recursion is usually very efficient in reliability computations and is used in other methodologies, such as the Universal Generating Function method \cite{L05}, factoring methods \cite{KZ03,TB17} or ad-hoc methods for particular systems, see \cite{MLAD15} for instance. In the case of the algebraic Sum of Disjoint Products, redundancy is avoided via efficient computation of involutive bases \cite{IPS22}.
The algebraic methods for computing system reliability based on monomial ideals are implemented in the \verb|C++| library \verb|CoCoALib| \cite{cocoalib} by means of an ad-hoc class described in \cite{BPS21}. 

We describe here three examples of implementation of the methods given in the previous sections, to demonstrate that the reliability of stable systems (among others) can be efficiently computed using the algebraic approach. All computations were performed in a Macintosh laptop with an \verb|M1| processor, and \verb|8GB RAM|.

\subsection{Binary system with multi-state components}
In our first experiment the goal is to compare the algebraic versions of Improved Inclusion-Exclusion formulas and Sum of Disjoint Products. We study a class of binary systems with multi-state components. Let $S$ be a system with $n$ components, each of which can be in $M+1$ levels of performance, $0,1,\dots,M$. The system works if at least $k$ components are working at level at least $1$ or if any of the components is working at level $M$. These systems are not stable unless $M=1$, but their corresponding reliability ideals are quasi-stable. Table \ref{table:example_k_n_plus_m} describes the results of performing the necessary computations for obtaining the reliability of these systems using the two algebraic methods mentioned in Section \ref{sec:algebraic}. We observe that when the maximal exponent of the elements of the ideal (i.e. the maximal level of the components) is low, then the involutive approach, i.e. sum of disjoint products offers better performance. However, when we have high levels, the resolution approach, i.e. improved inclusion-exclusion, performs better. In the table, column {\em size res.} indicates the number of terms of the inclusion-exclusion formula obtained using a free resolution, column {\em size inv.} indicates the size on the involutive basis of the ideal, and the last two columns indicate the time in seconds used to compute the corresponding resolutions and involutive bases respectively.
\begin{table}[h]
  \label{table:example_k_n_plus_m}
  \begin{tabular}{cccllrr}
    \toprule
    n&k&M&size res.&size inv.&time res.(s)& time inv.(s)\\
    \midrule
    10&2&2&9217&55&0.0285&0.0046\\
    10&2&6&9217&235&0.0284&0.0157\\
    10&4&2&215853&385&0.3329&0.0508\\
    10&4&6&86107&29485&0.1851&8.0373\\
    15&2&2&458753&120&0.6806&0.0081\\
    15&2&6&458753&540&0.6792&0.0338\\
    15&4&2&44759722&1940&90.7032&1.2330\\
    15&4&6&11927763&182540&22.2200&257.3160\\
      \bottomrule
\end{tabular}
 \caption{Size and time to compute a resolution and involutive basis for systems with $n$ components that work whenever $k$ work at level at least $1$, or at least one of them is working at level $M$. } 
\end{table}
\subsection{Improved Inclusion-Exclusion for stable systems}
In this experiment we compare the performance of the usual algebraic algorithm for IIE formulas based on Mayer-Vietoris trees \cite{S09}, which has a good general performance \cite{BPS21}, with an implementation of the Eliahou-Kervaire resolution for stable ideals, corresponding to stable systems. Since stable and strongly stable ideals are well studied objects in commutative algebra and algebraic geometry, we will use the algebraic geometry system \verb|Macaulay2| \cite{M2} to generate examples. We used the \verb|Macaulay2| package \verb|StronglyStableIdeals| \cite{AL19} that computes all strongly stable ideals with a given Hilbert polynomial. In our example, we used ideals in $20$ variables and used the mentioned package to generate all strongly stable ideals such that their Hilbert polynomial is $t^2+5t+2$. This is a set of $636$ ideals in $20$ variables. For each of these ideals we computed the minimal free resolution, i.e. the information needed to construct the algebraic improved inclusion-exclusion formulas in two different ways. One is using the Mayer-Vietoris tree implementation, and the second one is using the Eliahou-Kervaire symbols, see Section \ref{sec:resolutionStable}.

Figure \ref{fig:stronglyStableIdealsLog} shows the times taken by both algorithms to compute the ranks of the free resolution of these ideals, versus the size of the resolution i.e. the total sum of the ranks in it, which is the number of summands in the improved inclusion-exclusion formula. We see that the Eliahou-Kervaire approach is faster for this kind of ideals. One can also see that while the performance of the MVT algorithm has a very strict dependence on the size of the resolution, the Eliahou-Kervaire algorithm shows a lower slope and some variability.

\begin{figure}
\centering
\includegraphics[scale=0.5]{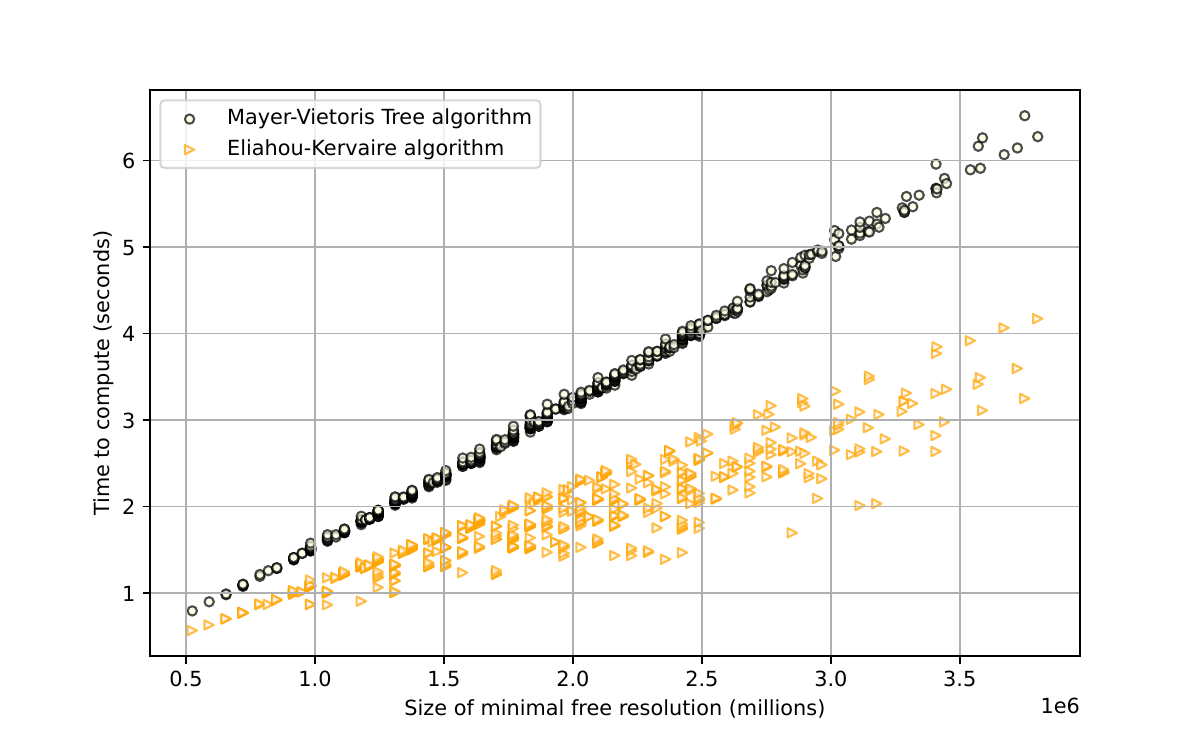}
\caption{Times taken by Mayer-Vietoris trees and Eliahou-Kervaire resolution algorithms to compute the ranks of the minimal free resolution of the $636$ strongly stable ideals in $20$ variables that have Hilbert polynomial $t^2+5t+2$.}
\label{fig:stronglyStableIdealsLog}
\end{figure}

\subsection{Binary $k$-out-of-$n$ and variants}
In this experiment we consider binary $k$-out-of-$n$  systems, and variants (consecutive linear and cyclic $k$-out-of-$n$ systems). Usual $k$-out-of-$n$ systems are strongly stable. For them, the Janet bases coincides with the minimal generating set, hence the size of their Mayer-Vietoris tree (and the Aramova-Herzog resolution) is typically much bigger than that of their involutive bases. The behavior of consecutive and cyclic $k$-out-of-$n$ systems is different in these respects. These are not stable systems, and therefore we cannot directly apply the Aramova-Herzog resolution. We can, however, still make use of Mayer-Vietoris trees and involutive bases. In this cases, the size of the Mayer-Vietoris tree is not big compared with the involutive basis, and therefore it will be the preferred method for computing reliability.

A comparison of the sizes of Mayer-Vietoris trees and Janet bases for $k$-out-of-$n$, consecutive $k$-out-of-$n$ and cyclic $k$-out-of-$n$ systems is given in Figure \ref{fig:k-out-of-n_variants}, for $n=10,15$ and $k=3,4,\dots,n$. The $y$-axis indicates, using a $\log$ scale, the ratio between the size of the Mayer-Vietoris tree with respect to the size of the corresponding Janet basis. The shaded region of the graph corresponds to the zone where this ratio is smaller than one, i.e. where the Janet basis is bigger than the Mayer-Vietoris tree. Observe that for the usual $k$-out-of-$n$ systems, which are stable, the Janet basis are much smaller than the corresponding Mayer-Vietoris trees.

\begin{figure}
\centering
\includegraphics[scale=0.4]{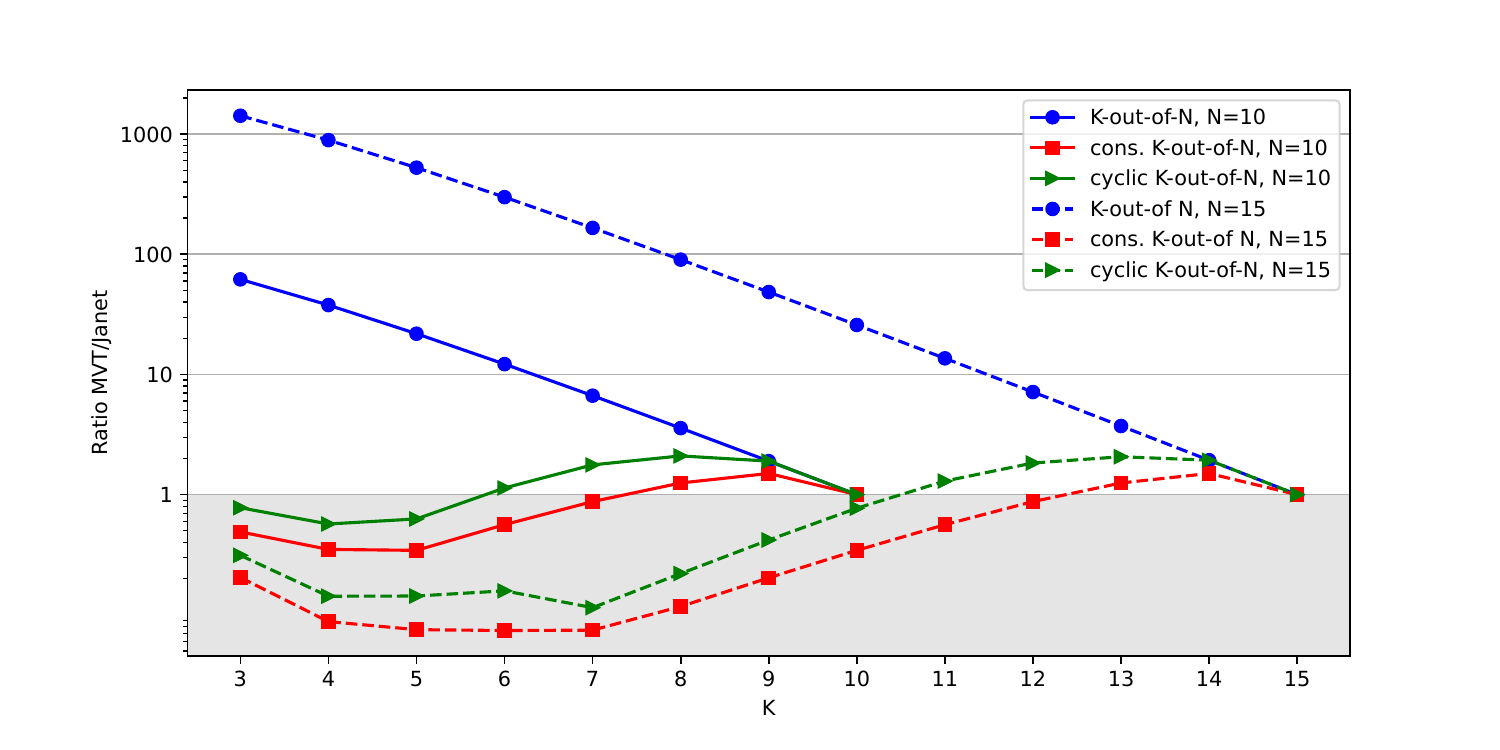}
\caption{Ratio between size of Mayer-Vietoris tree and Janet basis for k-out-of-n systems and variants (consecutive and cyclic). The shaded region is where the Janet basis are bigger than the Mayer-Vietoris trees.}
\label{fig:k-out-of-n_variants}
\end{figure}
\medskip

\noindent{\bf Acknowledgements.}
R. I, P. P-O, and E. S-d-C are partially supported by grant PID2020-116641GB-100 funded by MCIN/AEI/10.13039/501100011033. F.M. is supported by the UiT Aurora project MASCOT, KU Leuven grant iBOF/23/064, and FWO grants G0F5921N  (Odysseus) and G023721N.

%\subsection{A fully computed example}

%\noindent{\bf Data availability statement}
%All the data are supplied within the paper.
%\bibliographystyle{abbrv}
%\bibliography{bibliography.bib}
%\section*{ }
\printbibliography 
%\bibliographystyle{plain}
 %\bibliography{bibliography}

%\begin{acknowledgements}
%\end{acknowledgements}



%\begin{thebibliography}{}
% etc
%\end{thebibliography}

\medskip
\noindent 
\small{\textbf{Authors' addresses}

\medskip\medskip  \noindent
Departamento de Matem\'aticas y Computaci\'on, Universidad de La Rioja, Spain\\
           E-mail address: {\tt rodrigo.iglesias@unirioja.es} 
           
\medskip  \noindent
Department of Computer Science, KU Leuven, Celestijnenlaan 200A, B-3001 Leuven, Belgium\\ 
   Department of Mathematics, KU Leuven, Celestijnenlaan 200B, B-3001 Leuven, Belgium\\  UiT – The Arctic University of Norway, 9037 Troms\o, Norway\\
   E-mail address: {\tt fatemeh.mohammadi@kuleuven.be}

\medskip  \noindent
Departamento de Matem\'aticas y Computaci\'on, Universidad de La Rioja, Spain\\
           E-mail address: {\tt patricia.pascualo@unirioja.es} 

\medskip  \noindent
Departamento de Matem\'aticas y Computaci\'on, Universidad de La Rioja, Spain\\
           E-mail address: {\tt esaenz-d@unirioja.es}

\medskip  \noindent
% H. P. Wynn \at
	  London School of Economics and the Alan Turing Institute, London, UK\\
	  E-mail address: {\tt h.wynn@lse.ac.uk}


%\institute{R. Iglesias, P. Pascual-Ortigosa, E. S\'aenz-de-Cabez\'on \at             Departamento de Matem\'aticas y Computaci\'on, Universidad de La Rioja, Spain\\
             % Tel.: +123-45-678910\\
        %      \email{rodrigo.iglesias@unirioja.es, patricia.pascualo@unirioja.es, esaenz-d@unirioja.es} 
              %\email{patricia.pascualo@unirioja.es  \\       %  \\
    %       \and
   %        F. Mohammadi \at
	%  Department of Computer Science, KU Leuven, Belgium\\
%	  \email{fatemeh.mohammadi@kuleuven.be}
%	  \and 
%	  H. P. Wynn \at
%	  London School of Economics and the Alan Turing Institute, London, UK\\
%	  \email{h.wynn@lse.ac.uk}
%	  }


\end{document}